\documentclass[10pt,twoside]{article}
\usepackage{mathrsfs}
\usepackage{amsmath}
\usepackage{amssymb}
\usepackage{fancyhdr}
\usepackage{latexsym}
\usepackage{bbding}
\usepackage{mathrsfs}
\usepackage{wasysym}
\usepackage{multicol,graphics}

\newtheorem{theorem}{Theorem}[section]
\newtheorem{lemma}[theorem]{Lemma}

\newtheorem{definition}[theorem]{Definition}

\newtheorem{proposition}[theorem]{Proposition}
\numberwithin{equation}{section}
\newenvironment{proof}[1][Proof]{\noindent\textbf{#1.} }{\hfill $\Box$}
\allowdisplaybreaks

 \makeatletter
\begin{document}
\title{{The Optimal Temporal Decay Estimates for the Fractional Power Dissipative Equation in Negative Besov Spaces}}
\author{Jihong Zhao\\
[0.2cm] {\small College of Science, Northwest A\&F
University, Yangling, Shaanxi 712100,  P. R. China}\\
[0.2cm] {\small E-mail: jihzhao@163.com, zhaojih@nwsuaf.edu.cn}}
\date{}
\maketitle

\begin{abstract}
In this paper, we first generalize a new energy approach, developed by Y. Guo and Y. Wang \cite{GW12},   in the framework of homogeneous Besov spaces for proving the optimal temporal decay rates of  solutions to  the fractional power dissipative equation, then we apply this approach to the supercritical and critical quasi-geostrophic equation and the critical Keller-Segel system. We show that the negative Besov norm of solutions is preserved along time evolution,  and obtain the optimal temporal decay rates of the spatial derivatives of solutions by the Fourier splitting approach and the interpolation techniques.


\end{abstract}

\section{Introduction}
Consider the following Cauchy problem of the fractional power dissipative equation:
\begin{equation}\label{eq1.1}
\begin{cases}
\partial_{t}u+\Lambda^{\alpha} u=0\quad &\mbox{in}\ \
  \mathbb{R}^n\times(0,\infty),\\
u(x,0)=u_{0}(x)\quad &\mbox{in}\ \
  \mathbb{R}^n,\\
\end{cases}
\end{equation}
where $\alpha\in(0,2]$,  the fractional power of Laplacian $\Lambda^{\alpha}=(-\Delta)^{\frac{\alpha}{2}}$  is defined by
\begin{equation*}
  \Lambda^{\alpha}f(x):
    =2^{\alpha}\pi^{-\frac{n}{2}}\frac{\Gamma(\frac{n+\alpha}{2})}{\Gamma(-\frac{\alpha}{2})}P.V.\int_{\mathbb{R}^{n}}\frac{f(x-y)}{|y|^{n+\alpha}}dy.
\end{equation*}
We first  prove the following result:
\begin{theorem}\label{th1.1} Let $N\ge 0$ be an integer,  $s\ge0$ be a real number and $2\le p<\infty$.
If $u_{0}\in \dot{B}^{N}_{p,1}(\mathbb{R}^{n})\cap\dot{B}^{-s}_{p,\infty}(\mathbb{R}^{n})$, then there exists a constant $C_{0}$ such that
\begin{equation}\label{eq1.2}
  \|u(t)\|_{\dot{B}^{-s}_{p,\infty}}\le C_{0}.
\end{equation}
Moreover, for any
real number $\ell\in(-s,N]$, we have
\begin{equation}\label{eq1.3}
  \|u(t)\|_{\dot{B}^{\ell}_{p,1}}\le C_{0}(1+t)^{-\frac{\ell+s}{\alpha}}.
\end{equation}
\end{theorem}
\begin{proof}
 Applying the dyadic partition operator $\Delta_{j}$ to the equation \eqref{eq1.1}, we see that
\begin{equation*}
  \partial_{t}\Delta_{j}u+\Lambda^{\alpha}\Delta_{j}u=0,
\end{equation*}
which taking the standard $L^{2}$  inner product with $|\Delta_{j}u|^{p-2}\Delta_{j}u$ gives us to
\begin{equation}\label{eq1.4}
  \frac{1}{p}\frac{d}{dt}\|\Delta_{j}u\|_{L^{p}}^{p}+\int_{\mathbb{R}^{n}}\Lambda^{\alpha}\Delta_{j}u|\Delta_{j}u|^{p-2}\Delta_{j}udx=0.
\end{equation}
Thanks to the Bernstein's inequality (cf. \cite{CMZ07}, \cite{W05}), there exists a constant $\kappa$ such that
\begin{equation*}
  \int_{\mathbb{R}^{n}}\Lambda^{\alpha}\Delta_{j}u|\Delta_{j}u|^{p-2}\Delta_{j}udx\geq \kappa2^{\alpha j}\|\Delta_{j}u\|_{L^{p}}^{p}.
\end{equation*}
Thus, we obtain from \eqref{eq1.4} that
\begin{equation}\label{eq1.5}
   \frac{d}{dt}\|\Delta_{j}u\|_{L^{p}}+\kappa2^{\alpha j}\|\Delta_{j}u\|_{L^{p}}\leq 0.
\end{equation}
Multiplying the above inequality \eqref{eq1.5}  by $2^{j\ell}$, then taking $l^{\infty}$ norm to the resultant yields that
\begin{equation}\label{eq1.6}
  \frac{d}{dt}\|u\|_{\dot{B}^{\ell}_{p,\infty}}+\kappa\|u\|_{\dot{B}^{\ell+\alpha}_{p,\infty}}\leq 0,
\end{equation}
Integrating the above \eqref{eq1.6} in time, we obtain
\begin{equation}\label{eq1.7}
  \|u\|_{\dot{B}^{\ell}_{p,\infty}}\leq \|u_{0}\|_{\dot{B}^{\ell}_{p,\infty}}.
\end{equation}
This implies that inequality \eqref{eq1.2} holds by choosing $\ell=-s$. On the other hand,
 multiplying \eqref{eq1.5}  by $2^{j\ell}$, and taking  $l^{1}$ norm to the resultant yields that
\begin{equation}\label{eq1.8}
  \frac{d}{dt}\|u\|_{\dot{B}^{\ell}_{p,1}}+\kappa\|u\|_{\dot{B}^{\ell+\alpha}_{p,1}}\leq 0.
\end{equation}
Now for $-s<\ell\leq N$, we use the interpolation relation, see Lemma \ref{le2.4} below,  to get
\begin{equation*}
  \|u\|_{\dot{B}^{\ell}_{p,1}}\leq \|u\|_{\dot{B}^{-s}_{p,\infty}}^{\frac{\alpha}{\ell+s+\alpha}}\|u\|_{\dot{B}^{\ell+\alpha}_{p,1}}^{\frac{\ell+s}{\ell+s+\alpha}},
\end{equation*}
which combining \eqref{eq1.7} with $\ell=-s$ implies that
\begin{equation}\label{eq1.9}
  \|u\|_{\dot{B}^{\ell+\alpha}_{p,1}}\geq \|u_{0}\|_{\dot{B}^{-s}_{p,\infty}}^{-\frac{\alpha}{\ell+s}}\|u\|_{\dot{B}^{\ell}_{p,1}}^{1+\frac{\alpha}{\ell+s}}.
\end{equation}
Plugging \eqref{eq1.9} into \eqref{eq1.8}, we deduce that there exists a constant $C_{0}$ such that
\begin{equation*}
  \frac{d}{dt}\|u\|_{\dot{B}^{\ell}_{p,1}}+C_{0}\|u\|_{\dot{B}^{\ell}_{p,1}}^{1+\frac{\alpha}{\ell+s}}\leq 0.
\end{equation*}
Solving this inequality implies that
\begin{equation}\label{eq1.10}
  \|u\|_{\dot{B}^{\ell}_{p,1}}\leq \big(\|u_{0}\|_{\dot{B}^{\ell}_{p,1}}^{-\frac{\alpha}{\ell+s}}+\frac{C_{0}\alpha t}{\ell+s}\big)^{-\frac{\ell+s}{\alpha}}\leq C_{0}(1+t)^{-\frac{\ell+s}{\alpha}}.
\end{equation}
We complete the proof of Theorem \ref{th1.1}.
\end{proof}

\noindent\textbf{Remark 1.1} Theorem \ref{th1.1} is essentially inspired by Y. Guo and Y. Wang \cite{GW12}, where they developed a new energy
approach in the framework of Sobolev spaces for proving the optimal time decay rates of the solutions to the dissipative equation ($\alpha=2$).

\noindent\textbf{Remark 1.2} Theorem \ref{th1.1} generalizes the corresponding result in \cite{ZL15}, which we relax the regularity of the initial data in a wider range of Besov spaces. Moreover, the restrictive condition $p\ge2$ is due to Bernstein's inequality, which we don't know whether or not it is true for $1\le p<2$.

The structure of this paper is arranged as follows. In Section 2, we first recall some basic facts on Littlewood-Paley theory, then collect some important analytic tools used in this paper. In Section 3, we state our main results related to the optimal temporal decay estimates of the solutions to the supercritical and critical quasi-geostrophic equations and the critical Keller-Segel system, respectively. Section 4 is devoted to the proof of Theorem \ref{th3.1}, while Section 5 is devoted to the proof of Theorem \ref{th3.2}.

\section{Preliminaries}

\subsection{Notations}

Throughout this paper, we shall use the following notations.
\begin{itemize}
\item Let $T_1$, $T_2$ be two operators, we denote the commutator between
$T_1$ and $T_2$ by
$[T_1, T_2]:=T_1T_2-T_2T_1$.

\item  $f\lesssim g$ means that there
is a generic constant $C$ (always independent of $x, t$) such that $f\le Cg$.  $f\approx g$ means that $f\lesssim g$
and $g\lesssim f$.

\item We use $(f|g)$ to denote the standard
$L^{2}(\mathbb{R}^{n})$ inner product of two functions $f$ and $g$.

\item For a quasi-Banach space $X$ and for any $0<T\leq\infty$, we use standard notation $L^{p}(0,T; X)$ or $L^{p}_{T}(X)$  for the quasi-Banach space of Bochner measurable functions
$f$ from $(0, T)$ to $X$ endowed with the norm
\begin{equation*}
\|f\|_{L^{p}_{T}(X)}:=
\begin{cases}
  (\int_{0}^{T}\|f(\cdot,t)\|_{X}^{p}dt)^{\frac{1}{p}}\ \ \ &\text{for}\ \ \ 1\leq p<\infty,\\
  \sup_{0\leq t\leq T}\|f(\cdot,t)\|_{X}\ \ \ &\text{for}\ \ \ p=\infty.
\end{cases}
\end{equation*}
In particular, if $T=\infty$, we use $\|f\|_{L^{p}_{t}(X)}$ instead of $\|f\|_{L^{p}_{\infty}(X)}$.

\item $(d_{j})_{j\in\mathbb{Z}}$ will be a generic element of
$l^{1}(\mathbb{Z})$ so that $d_{j}\ge0$ and
$\sum_{j\in\mathbb{Z}}d_{j}=1$.

\end{itemize}

\subsection{Littlewood-Paley theory and Besov spaces}

We start with the Fourier transform. Let
$\mathcal{S}(\mathbb{R}^{n})$ be the Schwartz class of rapidly
decreasing function, and $\mathcal{S}'(\mathbb{R}^{n})$ of temperate distributions be the dual set of $\mathcal{S}(\mathbb{R}^{n})$. Given  $f\in\mathcal{S}(\mathbb{R}^{n})$, the Fourier transform $\mathcal{F}(f)$ (or $\hat{f}$) is defined by
$$
  \mathcal{F}(f)(\xi)=\int_{\mathbb{R}^{n}}e^{-2\pi i x\cdot\xi}f(x)dx.
$$
More generally, the Fourier transform of  a tempered distribution $f\in\mathcal{S}'(\mathbb{R}^{n})$ is defined by the dual argument in the standard way.

We now introduce a dyadic decomposition in $\mathbb{R}^{n}$.
Let $\varphi\in\mathcal{S}(\mathbb{R}^{n})$ be a  smooth radial function valued in $[0,1]$ such that  $\varphi$ is supported in the shell $\mathcal{C}=\{\xi\in\mathbb{R}^{n},\ \frac{3}{4}\leq
|\xi|\leq\frac{8}{3}\}$, and
\begin{align*}
  \sum_{j\in\mathbb{Z}}\varphi(2^{-j}\xi)=1, \ \ \ \forall\xi\in\mathbb{R}^{n}\backslash\{0\}.
\end{align*}
Let $h=\mathcal{F}^{-1}\varphi$. Then  for any $f\in\mathcal{S}'(\mathbb{R}^{n})$, we define the dyadic blocks
$\Delta_{j}$ and $S_{j}$ as follows:
\begin{align}\label{eq2.1}
 \Delta_{j}f: =2^{nj}\int_{\mathbb{R}^{n}}h(2^{j}y)f(x-y)dy \ \ \ \text{and}\ \ \
  S_{j}f:=\sum_{k\leq j-1}\Delta_{k}f.
\end{align}
By telescoping the series, we have the following homogeneous Littlewood-Paley decomposition:
\begin{equation*}
  f=\sum_{j\in\mathbb{Z}}\Delta_{j}f \ \ \text{for}\ \
  f\in\mathcal{S}'(\mathbb{R}^{n})/\mathcal{P}(\mathbb{R}^{n}),
\end{equation*}
where $\mathcal{P}(\mathbb{R}^{n})$ is the set of polynomials (see \cite{BCD11}).
We remark here that the Littlewood-Paley decomposition satisfies the property of almost orthogonality, that is to say, for any $f, g\in\mathcal{S}'(\mathbb{R}^{n})/\mathcal{P}(\mathbb{R}^{n})$, the following properties hold:
\begin{align}\label{eq2.2}
  \Delta_{i}\Delta_{j}f\equiv0\ \ \ \text{if}\ \ \ |i-j|\geq 2\ \ \ \text{and}\ \ \
  \Delta_{i}(S_{j-1}f\Delta_{j}g)\equiv0 \ \ \ \text{if}\ \ \ |i-j|\geq 5.
\end{align}
For more details, see \cite{BCD11} and \cite{T83}.

Using the above decomposition, the  stationary/time dependent homogeneous Besov spaces can be defined as follows:

\begin{definition}\label{de2.1}
For $s\in \mathbb{R}$, $1\leq p,r\leq\infty$ and $f\in\mathcal{S}'(\mathbb{R}^{n})$, we set
\begin{equation*}
  \|f\|_{\dot{B}^{s}_{p,r}}:= \begin{cases} \left(\sum_{j\in\mathbb{Z}}2^{jsr}\|\Delta_{j}f\|_{L^{p}}^{r}\right)^{\frac{1}{r}}
  \ \ &\text{for}\ \ 1\leq r<\infty,\\
  \sup_{j\in\mathbb{Z}}2^{js}\|\Delta_{j}f\|_{L^{p}}\ \
  &\text{for}\ \
  r=\infty.
 \end{cases}
\end{equation*}
Then the homogeneous Besov
space $\dot{B}^{s}_{p,r}(\mathbb{R}^{n})$ is defined by
\begin{itemize}
\item For $s<\frac{n}{p}$ (or $s=\frac{n}{p}$ if $r=1$), we define
\begin{equation*}
  \dot{B}^{s}_{p,r}(\mathbb{R}^{n}):=\Big\{f\in \mathcal{S}'(\mathbb{R}^{n}):\ \
  \|f\|_{\dot{B}^{s}_{p,r}}<\infty\Big\}.
\end{equation*}
\item If $k\in\mathbb{N}$ and $\frac{n}{p}+k\leq s<\frac{n}{p}+k+1$ (or $s=\frac{n}{p}+k+1$ if $r=1$), then $\dot{B}^{s}_{p,r}(\mathbb{R}^{n})$
is defined as the subset of distributions $f\in\mathcal{S}'(\mathbb{R}^{n})$ such that $\partial^{\beta}f\in\mathcal{S}'(\mathbb{R}^{n})$
whenever $|\beta|=k$.
\end{itemize}
\end{definition}

\begin{definition}\label{de2.2}  For $0<T\leq\infty$, $s\leq \frac{n}{p}$ (resp. $s\in \mathbb{R}$) and
$1\leq p, r, \rho\leq\infty$. We define the mixed time-space $\mathcal{L}^{\rho}(0,T; \dot{B}^{s}_{p,r}(\mathbb{R}^{n}))$
as the completion of $\mathcal{C}([0,T]; \mathcal{S}(\mathbb{R}^{n}))$ by the norm
$$
  \|f\|_{\mathcal{L}^{\rho}_{T}(\dot{B}^{s}_{p,r})}:=\left(\sum_{j\in\mathbb{Z}}2^{jsr}\left(\int_{0}^{T}
  \|\Delta_{j}f(\cdot,t)\|_{L^{p}}^{\rho}dt\right)^{\frac{r}{\rho}}\right)^{\frac{1}{r}}<\infty
$$
with the usual change if $\rho=\infty$
or $r=\infty$.  For simplicity, we use $\|f\|_{\mathcal{L}^{\rho}_{t}(\dot{B}^{s}_{p,r})}$ instead of $\|f\|_{\mathcal{L}^{\rho}_{\infty}(\dot{B}^{s}_{p,r})}$.
\end{definition}

The following properties of Besov spaces are well-known:

 (1) If  $s<\frac{n}{p}$ or
$s=\frac{n}{p}$ and $r=1$, then
$(\dot{B}^{s}_{p,r}(\mathbb{R}^{n}),\|\cdot\|_{\dot{B}^{s}_{p,r}})$
is a Banach space which is continuously embedded in
$\mathcal{S}'(\mathbb{R}^{n})$.

(2) In the case that $p=r=2$, we get the homogeneous Sobolev space
$\dot{H}^{s}(\mathbb{R}^{n})\cong\dot{B}^{s}_{2,2}(\mathbb{R}^{n})$,
which is endowed with the equivalent norm
$\|f\|_{\dot{H}^{s}}=\|\Lambda^{s}f\|_{L^{2}}$  with $\Lambda=\sqrt{-\Delta}$.

(3) Let $s\in \mathbb{R}$, $1\leq p,r\leq\infty$, and
$u\in\mathcal{S}'(\mathbb{R}^{n})/\mathcal{P}(\mathbb{R}^{n})$. Then $u\in
\dot{B}^{s}_{p,r}(\mathbb{R}^{n})$ if and only if there exists
$\{d_{j,r}\}_{j\in\mathbb{Z}}$ such that $d_{j,r}\ge0$,
$\|d_{j,r}\|_{l^{r}}=1$ and
$$
  \|\Delta_{j}u\|_{L^{p}}\lesssim
  d_{j,r}2^{-js}\|u\|_{\dot{B}^{s}_{p,r}} \ \ \text{for all }\
  j\in\mathbb{Z}.
$$

(4)  According to the  Minkowski inequality,
it is readily to see that
\begin{equation*}
\begin{cases}
  \|f\|_{\mathcal{L}^{\rho}_{T}(\dot{B}^{s}_{p,r})}\leq\|f\|_{L^{\rho}_{T}(\dot{B}^{s}_{p,r})} \ \ \  \text{if}\ \ \  \rho\leq r,\\
  \|f\|_{L^{\rho}_{T}(\dot{B}^{s}_{p,r})}\leq \|f\|_{\mathcal{L}^{\rho}_{T}(\dot{B}^{s}_{p,r})} \ \ \ \text{if} \ \ \  r\leq \rho.
\end{cases}
\end{equation*}

\subsection{Useful analytic tools}

For the convenience of the reader, we present some crucial analytic tools as follows. The first one is an improved Bernstein inequalities, see, for example, \cite{BCD11, W05}.

\begin{lemma}\label{le2.3} {\em (\cite{BCD11},  \cite{W05})}
Let $\mathcal{B}$ be a ball, and $\mathcal{C}$  a ring in
$\mathbb{R}^{n}$. There exists a constant $C$ such that for any
positive real number $\lambda$, any nonnegative integer $k$ and any
couple of real numbers $(a,b)$ with $1\leq a\le b\leq \infty$, we
have
\begin{equation}\label{eq2.3}
   \operatorname{supp}\hat{f}\subset\lambda\mathcal{B}\ \ \Rightarrow\ \   \sup_{|\alpha|=k}\|\Lambda^{\alpha}f\|_{L^{b}}\leq
   C^{k+1}\lambda^{k+n(\frac{1}{a}-\frac{1}{b})}\|f\|_{L^{a}},
\end{equation}
\begin{equation}\label{eq2.4}
   \operatorname{supp}\hat{f}\subset\lambda\mathcal{C} \ \ \Rightarrow\ \   C^{-1-k}\lambda^{k}\|f\|_{L^{a}}\leq
   \sup_{|\alpha|=k}\|\Lambda^{\alpha}f\|_{L^{a}}\leq  C^{1+k}\lambda^{k}\|f\|_{L^{a}}.
\end{equation}
\end{lemma}

Secondly, we present  some basic properties of Besov spaces (see \cite{BCD11}, \cite{T83}).

\begin{lemma}\label{le2.4}  {\em (\cite{BCD11}, \cite{T83})}
The following properties hold:
\begin{itemize}
\item [i)] Density: The set $C_{0}^{\infty}(\mathbb{R}^{n})$ is dense in $\dot{B}^{s}_{p,r}(\mathbb{R}^{n})$ if $|s|<\frac{n}{p}$ and $1\leq p,r<\infty$ or $s=\frac{n}{p}$ and $r=1$.

\item [ii)] Derivatives: There exists a universal constant $C$ such that
\begin{equation*}
    C^{-1}\|u\|_{\dot{B}^{s}_{p,r}}\leq \|\nabla u\|_{\dot{B}^{s-1}_{p,r}}\leq C\|u\|_{\dot{B}^{s}_{p,r}}.
\end{equation*}

\item [iii)] Fractional derivatives: Let $\Lambda=\sqrt{-\Delta}$ and $\sigma\in\mathbb{R}$. Then the operator $\Lambda^{\sigma}$ is an isomorphism from $\dot{B}^{s}_{p,r}(\mathbb{R}^{n})$
to $\dot{B}^{s-\sigma}_{p,r}(\mathbb{R}^{n})$.

\item [iv)] Algebraic properties: For $s>0$, $\dot{B}^{s}_{p,r}(\mathbb{R}^{n})\cap L^{\infty}(\mathbb{R}^{n})$ is an algebra.  Moreover,  $\dot{B}^{\frac{n}{p}}_{p,1}(\mathbb{R}^{n})\hookrightarrow \dot{B}^{0}_{\infty,1}(\mathbb{R}^{n})\hookrightarrow L^{\infty}(\mathbb{R}^{n})$, and for any $f,g\in\dot{B}^{s}_{p,r}(\mathbb{R}^{n})\cap L^{\infty}(\mathbb{R}^{n})$, we have
$$
  \|fg\|_{\dot{B}^{s}_{p,r}}\leq  \|f\|_{\dot{B}^{s}_{p,r}}\|g\|_{L^{\infty}}+ \|g\|_{\dot{B}^{s}_{p,r}}\|f\|_{L^{\infty}}.
$$

\item [v)] Imbedding: For $1\leq p_{1}\leq p_{2}\leq \infty$ and $1\leq r_{1}\leq r_{2}\leq \infty$, we have the continuous imbedding  $\dot{B}^{s}_{p_{1},r_{1}}(\mathbb{R}^{n})\hookrightarrow \dot{B}^{s-n(\frac{1}{p_{1}}-\frac{1}{p_{2}})}_{p_{2},r_{2}}(\mathbb{R}^{n})$.

\item [vi)] Interpolation:  For  $s_{1},s_{2}\in \mathbb{R}$ such that $s_{1}<s_{2}$ and $\vartheta\in (0,1)$,  there exists a constant $C$
such that
\begin{align*}
    &\|u\|_{\dot{B}^{s_{1}\vartheta+s_{2}(1-\vartheta)}_{p,r}}\leq C\|u\|_{\dot{B}^{s_{1}}_{p,r}}^{\vartheta}\|u\|_{\dot{B}^{s_{2}}_{p,r}}^{1-\vartheta},\\
   &\|u\|_{\dot{B}^{s_{1}\vartheta+s_{2}(1-\vartheta)}_{p,1}}\leq
   \frac{C}{s_{2}-s_{1}}\left(\frac{1}{\vartheta}+\frac{1}{1-\vartheta}\right)\|u\|_{\dot{B}^{s_{1}}_{p,\infty}}^{\vartheta}
   \|u\|_{\dot{B}^{s_{2}}_{p,\infty}}^{1-\vartheta}.
\end{align*}
\end{itemize}
\end{lemma}

Finally we recall the following Bony's paradifferential decomposition
 (see \cite{B81}). The
paraproduct between $f$ and $g$ is defined by
\begin{equation*}
  T_{f}g:=\sum_{j\in\mathbb{Z}}S_{j-1}f\Delta_{j}g.
\end{equation*}
Thus we have the formal decomposition
\begin{equation*}
  fg=T_{f}g+T_{g}f+R(f,g),
\end{equation*}
where
\begin{equation*}
  R(f,g):=\sum_{j\in\mathbb{Z}}\Delta_{j}f\widetilde{\Delta_{j}}g \ \
  \text{and}\ \
  \widetilde{\Delta_{j}}:=\Delta_{j-1}+\Delta_{j}+\Delta_{j+1}.
\end{equation*}

\section{Main results}

\subsection{Surface quasi-geostrophic equation}

Consider the Cauchy problem of the dissipative surface quasi-geostrophic
equation
\begin{equation}\label{eq3.1}
\begin{cases}
  \partial_{t}\theta+\mathbf{u}\cdot\nabla\theta+\mu\Lambda^{\alpha}\theta=0\quad &\mbox{in}\ \
  \mathbb{R}^2\times(0,\infty),\\
  \theta(x,0)=\theta_{0}(x)\quad &\mbox{in}\ \
  \mathbb{R}^2,\\
\end{cases}
\end{equation}
where $\alpha\in(0, 2]$ and $\mu\geq0$ are parameters,  $\theta$ is an unknown scalar function
representing the potential temperature,
$\mathbf{u}$ is the fluid
velocity field determined by
$$
 \mathbf{u}=(u^{1},u^{2})=(-\mathcal{R}_{2}\theta,\mathcal{R}_{1}\theta),
$$
where $\mathcal{R}_{j}$ ($j=1,2$) are 2D Riesz
transforms whose symbols are given by
$\frac{i\xi_{j}}{|\xi_{j}|}$. Since the concrete value of the constant $\mu$ plays no role in our discussion,
for simplicity, we assume that $\mu=1$ throughout this paper.

The inviscid surface quasi-geostrophic equation \eqref{eq3.1} ($\mu=0$) was first introduced by Constantin, Majda and Tabak \cite{CMT94} to model frontogenesis in meteorology, a formation of sharp fronts between masses of hot
and cold air, then it becomes
an important model in geophysical fluid dynamics used in meteorology
and oceanography, see Pedlosky \cite{P87}. In last two decades, the global regularity issue to  the equation
\eqref{eq3.1} has attracted enormous attention, and  many
remarkable results have been obtained. Generally speaking, the study of the equation
\eqref{eq3.1} is divided into three cases: the subcritical case
($1<\alpha\leq2$), the critical case ($\alpha=1$), and the supercritical
case ($0<\alpha<1$). For the subcritical case $1<\alpha\leq2$, the problem
is more or less resolved: Constantin and Wu \cite{CW99} established
global regularity of weak solutions to the equation \eqref{eq3.1}
with smooth initial data, see also Resnick \cite{R95}. For the
critical case $\alpha=1$, the problem was first considered by Constantin, C\'{o}rdoba and Wu
\cite{CCW01}, where the unique global solution with small initial data was
proved, subsequently it was successfully addressed by the following two mathematical groups:
Kiselev, Nazarov and Volberg \cite{KNV07} proved global
well-posedness of the equation \eqref{eq3.1} with periodic $C^{\infty}$
data by using a certain non-local maximum principle for a suitable
chosen modulus of continuity; Caffarelli and Vasseur \cite{CV10}
obtained a global regular weak solution for equation \eqref{eq3.1}
with merely $L^{2}$ initial data by using the modified De Georgi
interation.

As far as the author is concerned,  the global regularity issue for the supercritical case $\alpha<1$ remains open.  We mention that Constantin and Wu \cite{CW08, CW09} proved that if the solution of  the equation \eqref{eq3.1} is
in the H\"{o}lder spapce $C^{\delta}$ with $\delta>1-2\alpha$, then the solution is actually smooth.  This result was subsequently extended by Dong and
Pavlovic \cite{DP09} to cover the case  $\delta=1-2\alpha$.  We refer the reader to see \cite{D11} and \cite{S10} for some eventual regularity results of the
equation \eqref{eq3.1} with supercritical dissipation.

Note that the surface quasi-geostrophic equation \eqref{eq3.1} has a scaling. Indeed,  it is easy to see that if the pair $(\theta, \mathbf{u})$ solves the  equation \eqref{eq3.1} with initial data $\theta_{0}$, then the pair $(\theta_{\lambda}, \mathbf{u}_{\lambda})$ with
\begin{equation*}
 \theta_{\lambda}(x,t):=\lambda^{\alpha-1}\theta(\lambda x,
  \lambda^{\alpha}t), \ \ \ \mathbf{u}_{\lambda}(x,t):=\lambda^{\alpha-1}\mathbf{u}(\lambda x,
  \lambda^{\alpha}t),
\end{equation*}
is also a solution to the equation \eqref{eq3.1}  with initial data  $\theta_{0\lambda}(x):=\lambda^{\alpha-1} \theta_{0}(\lambda x)$. In particular, the norm of $\theta_{0}\in\dot{B}^{1+\frac{2}{p}-\alpha}_{p,1}(\mathbb{R}^{2})$ ($1\leq p\leq\infty$) is scaling invariant under the above change of scale.
Cheng, Miao and Zhang \cite{CMZ07} and  Hmidi and  Keraani \cite{HK07}, respectively,  proved global well-posedness of the equation \eqref{eq3.1} both in
critical and supercritical dissipation with small initial data in the critical Besov space $\dot{B}^{1+\frac{2}{p}-\alpha}_{p,1}(\mathbb{R}^2)$. The limit case $p=\infty$ was completely tackled by  Abidi and Hmidi \cite{AH08} and Wang and Zhang \cite{WZ11a}, respectively.  For more interesting results related to this topic, we refer the reader to see \cite{BB15, CL03, CZ07, CC04, LY14, SVZ13}.

Motivated by the optimal time decay rates of the solutions to the fractional power dissipation equation \eqref{eq1.1} in the framework of homogeneous Besov spaces, we
aim at using this approach to the dissipative supercritical and critical surface quasi-geostrophic equation \eqref{eq3.1}.  The
main result is as follows:

\begin{theorem}\label{th3.1}
Let $\alpha\in(0,1]$  and $p\in[2, \infty)$.
Suppose that $\theta_{0}\in
\dot{B}^{1+\frac{2}{p}-\alpha}_{p,1}(\mathbb{R}^{2})$. Then  there exists a positive constant $\varepsilon$  such that  for any
$ \|\theta_{0}\|_{\dot{B}^{1+\frac{2}{p}-\alpha}_{p,1}}<\varepsilon$,
the equation \eqref{eq3.1} has a unique global solution $\theta$, which belongs to
\begin{equation*}
  \theta\in C([0,\infty),\dot{B}^{1+\frac{2}{p}-\alpha}_{p,1}(\mathbb{R}^{2}))\cap\mathcal{L}^{\infty}(0, \infty;
  \dot{B}^{1+\frac{2}{p}-\alpha}_{p,1}(\mathbb{R}^{2}))\cap  L^{1}(0, \infty;
  \dot{B}^{1+\frac{2}{p}}_{p,1}(\mathbb{R}^{2})).
\end{equation*}
If we assume further that $\theta_{0}\in
\dot{B}^{-s}_{r,\infty}(\mathbb{R}^{2})$ with  $2\le r\le p$,  $-\frac{2}{p}<s<1+\frac{2}{p}$,
then  there exists a constant $C_{0}$ such that for all $t\geq 0$,
\begin{equation}\label{eq3.2}
     \|\theta(t)\|_{\dot{B}^{-s}_{r,\infty}}\leq C_{0}.
\end{equation}
Moreover,  for any $\ell\in[-s-2(\frac{1}{r}-\frac{1}{p}), 1+\frac{2}{p}-\alpha]$, we have
\begin{equation}\label{eq3.3}
     \|\theta(t)\|_{\dot{B}^{\ell}_{r,1}}\leq C_{0}(1+t)^{-(\frac{\ell+s}{\alpha})-\frac{2}{\alpha}(\frac{1}{r}-\frac{1}{p})}.
\end{equation}
\end{theorem}

\noindent\textbf{Remark 3.1} An important feature in Theorem \ref{th3.1} is that the negative Besov norm of the solution is preserved along the time evolution, see Proposition \ref{pro4.3} below. Moreover,  we do not need to impose on small condition to the $\dot{B}^{-s}_{r,\infty}$-norm of  initial data,  which  enhances the time decay rates of the solution with the factor $\frac{s}{\alpha}$.

\noindent\textbf{Remark 3.2}  The general $L^{r}$ temporal decay rates of the solution can be obtained by the imbedding theory, for instance,  for any $2\le r<\infty$,
\begin{align*}
  \|\theta(t)\|_{L^{r}}\leq C\|\theta(t)\|_{\dot{H}^{1-\frac{2}{r}}}\leq C\|\theta(t)\|_{\dot{B}^{1-\frac{2}{r}}_{2,1}}\leq C(1+t)^{-\frac{s}{\alpha}-\frac{2}{\alpha}(1-\frac{1}{r}-\frac{1}{p})}.
\end{align*}

\noindent\textbf{Remark 3.3} We are not intending to tackle with the subcritical case $1<\alpha\le 2$ because of similar result as that of  Theorem \ref{th3.1} still holds in this case, but the proof is more or less standard.

The proof of Theorem \ref{th3.1} will be given in Section 4.

\subsection{Fractional Keller-Segel system}

Consider the following Cauchy problem of nonlinear nonlocal evolution system generalizing
the well-known Keller-Segel model of chemotaxis:
\begin{equation}\label{eq3.4}
\begin{cases}
  \partial_{t}u+\nu\Lambda^{\alpha}u+\nabla\cdot(u\nabla \psi)=0\quad &\mbox{in}\ \
  \mathbb{R}^2\times(0,\infty),\\
   -\Delta \psi=u\quad &\mbox{in}\ \
  \mathbb{R}^2\times(0,\infty),\\
  u(x,0)=u_0(x) \ \  &\mbox{in}\ \ \mathbb{R}^2.
\end{cases}
\end{equation}
where $\alpha\in(0,2]$ and $\nu>0$ are parameters,  $u$ and $\psi$ are two unknown functions which stand for the cell density and the concentration of the chemical attractant, respectively. For the sake of simplicity, we assume that $\nu=1$.

Of course, when $\alpha=2$, the system \eqref{eq3.4} is  a famous  biological model of chemotaxis, which is formulated  by E.F. Keller and L.A. Segel \cite{KS70} to describe  the collective motion of cells
under chemotactic attraction, leading possibly to aggregation of cells. It is well-known that the system \eqref{eq3.4} admits  finite time blowup solutions for large enough initial data, we refer the reader to see \cite{B98, BKZ14, BCM07, BDP06, HV961, HV962, JL92, LR091, LR092, NSY97} and the references therein for a comprehensive review of this topic.  On the other hand, when $\alpha=2$,  the solvability of the system \eqref{eq3.4}  with small initial data in various
classes of functions and distributions has been relatively well-developed,  for instance,  the Lebesgue space $L^{1}(\mathbb{R}^{n})\cap L^{\frac{n}{2}}(\mathbb{R}^{n})$  by Corrias, Perthame and Zaag \cite{CPZ04}, the Sobolev space $L^{1}(\mathbb{R}^{n})\cap W^{2,2}(\mathbb{R}^{n})$ by Kozono and Sugiyama \cite{KS08},  the Hardy space $\mathcal{H}^{1}(\mathbb{R}^{2})$ by Ogawa and Shimizu \cite{OS08}, the Besov space $\dot{B}^{0}_{1,2}(\mathbb{R}^{2})$ by Ogawa and Shimizu \cite{OS10},  the Besov space $\dot{B}^{-2+\frac{n}{p}}_{p,\infty}(\mathbb{R}^{n})$ and Fourier-Herz space $\dot{\mathcal{B}}^{-2}_{2}(\mathbb{R}^{n})$ by Iwabuchi \cite{I11}, for more results, see \cite{L13}.

For general fractional diffusion case $1\le\alpha<2$, the system \eqref{eq3.4} was first studied by Escudero in \cite{E06},  where it was used to describe the spatiotemporal distribution of a population density of random walkers undergoing L\'{e}vy flights, and the author proved that the one-dimensional system \eqref{eq3.4} possesses global in time solutions not
only in the case of $\alpha=2$ but also in the case $1<\alpha<2$. Since the fractional Keller-Segel system \eqref{eq3.4} is also scaling invariant under the following change of scale:
\begin{equation*}
  u_{\lambda}(x,t):=\lambda^{\alpha}u(\lambda x, \lambda^{\alpha}t)\ \ \ \text{and}\ \ \   \psi_{\lambda}(x,t):=\lambda^{\alpha-2}\psi(\lambda x, \lambda^{\alpha}t),
\end{equation*}
 the global well-posedness with small initial data in different scaling invariant spaces (so-called critical spaces) has been considerably established, for example, the critical Lebesgue space $L^{\frac{n}{\alpha}}(\mathbb{R}^{n})$  with $1<\alpha<2$ by Biler and Karch \cite{BK10},  the critical  Besov spaces $\dot{B}^{-\alpha+\frac{2}{p}}_{p,q}(\mathbb{R}^{2})$ with $1<\alpha<2$ by Biler and Wu \cite{BW09} and   Zhai \cite{Z10}, the critical Fourier-Herz
space $\mathcal{\dot{B}}^{2-2\alpha}_{q}(\mathbb{R}^{n})$ with $1<\alpha\leq2$ by Wu and Zheng \cite{WZ11b}.  Recently, in one dimensional space, the authors in \cite{BG151}  showed that the solution to the critical Keller-Segel system \eqref{eq3.4} ($\alpha=1$) on $\mathbb{S}^1$ remains smooth for any initial data and any positive time, moreover, they studied the global existence of solutions to a one-dimensional critical Keller-Segel system with logistic term, see \cite{BG152}.

Note that in the supercritical case $0<\alpha<1$,  since the dissipative term $\Lambda^{\alpha}u$ is not strong enough to dominate the nonlinear nonlocal term $\nabla\cdot(u\nabla\psi)$, the well-posedness issue of the system \eqref{eq3.4} in dimensions two is still an open problem. On the other hand, in the critical case $\alpha=1$, the author in this paper has successfully proved the global well-posedness of the system \eqref{eq3.4} with small initial data $u_{0}\in \dot{B}^{-1+\frac{2}{p}}_{p,1}(\mathbb{R}^2)$ ($1\leq p\leq \infty$), see \cite{Z15} for more results.  Motivated by this result, applying the approach illustrated in Theorem \ref{th1.1}, the optimal time decay rates of the solutions to the critical Keller-Segel system \eqref{eq3.4} in the homogeneous Besov spaces can be proved. The main result is as follows:

\begin{theorem}\label{th3.2}
Let  $\alpha=1$ and  $p\in[2,\infty)$.
Suppose that $u_{0}\in
\dot{B}^{-1+\frac{2}{p}}_{p,1}(\mathbb{R}^{2})$. Then  there exists a positive constant $\varepsilon$  such that  for any
$ \|u_{0}\|_{\dot{B}^{-1+\frac{2}{p}}_{p,1}}<\varepsilon$,
the system \eqref{eq3.4} has a unique global solution $u$, which belongs to
\begin{equation*}
  u\in C([0,\infty),\dot{B}^{-1+\frac{2}{p}}_{p,1}(\mathbb{R}^{2}))\cap\mathcal{L}^{\infty}(0, \infty;
  \dot{B}^{-1+\frac{2}{p}}_{p,1}(\mathbb{R}^{2}))\cap  L^{1}(0, \infty;
  \dot{B}^{\frac{2}{p}}_{p,1}(\mathbb{R}^{2})).
\end{equation*}
If we assume further that $u_{0}\in
\dot{B}^{-s}_{r,\infty}(\mathbb{R}^{2})$ with  $2\le r\le p$,  $1-\frac{2}{p}<s<1+\frac{2}{p}$,
then  there exists a constant $C_{0}$ such that for all $t\geq 0$,
\begin{equation}\label{eq3.5}
     \|u(t)\|_{\dot{B}^{-s}_{r,\infty}}\leq C_{0}.
\end{equation}
Moreover,  for any $\ell\in[-s-2(\frac{1}{r}-\frac{1}{p}), -1+\frac{2}{p}]$,  we have
\begin{equation}\label{eq3.6}
     \|u(t)\|_{\dot{B}^{\ell}_{r,1}}\leq C_{0}(1+t)^{-(\ell+s)-2(\frac{1}{r}-\frac{1}{p})}.
\end{equation}
\end{theorem}

\noindent\textbf{Remark 3.4}  We mention here that similar result as that of  Theorem \ref{th3.2} still holds to the subcritical Keller-Segel system ($1<\alpha\le 2$).

The proof of Theorem \ref{th3.2} will be given in Section 5.

\section{Proof of Theorem  \ref{th3.1}}

Notice that the global well-posedness part with small initial data has been proved by \cite{AH08} and \cite{CMZ07}, respectively, thus we need only to prove the temporal decay part in Theorem \ref{th3.1}.  We first aim at establishing two basic energy inequalities in the framework of homogeneous Besov spaces, then prove the decay estimates
\eqref{eq3.2} and \eqref{eq3.3}  by using the approach illustrated in Theorem \ref{th1.1},  the Fourier splitting approach and the interpolation theory.

\subsection{Basic energy inequalities}
Let
$$
  \mathcal{E}(t):= \|\theta(t)\|_{\dot{B}^{1+\frac{2}{p}-\alpha}_{p,1}}\ \ \ \text{and}\ \ \ \mathcal{Y}(t):=\int_{0}^{t} \|\theta(\tau)\|_{\dot{B}^{1+\frac{2}{p}}_{p,1}}d\tau.
$$

\begin{proposition}\label{pro4.1}
Under the assumptions of Theorem \ref{th3.1}, let $\theta$ be the solution of  the equation \eqref{eq3.1} corresponding to the initial data $\theta_{0}$.  Then there exist two constants $\kappa$ and $K$ such that the following inequality holds:
\begin{equation}\label{eq4.1}
   \frac{d}{dt}(e^{-K\mathcal{Y}(t)}\mathcal{E}(t))+\kappa e^{-K\mathcal{Y}(t)}\mathcal{Y}'(t)\leq 0.
\end{equation}
\end{proposition}

To prove Proposition \ref{pro4.1}, we set
\begin{equation*}
  \widetilde{\theta}(t):=e^{-K\mathcal{Y}(t)}\theta(t)\ \ \ \text{and}\ \ \ \widetilde{\mathbf{u}}(t):= e^{-K\mathcal{Y}(t)}\mathbf{u}(t)=(-\mathcal{R}_{2}\widetilde{\theta}(t),\mathcal{R}_{1}\widetilde{\theta}(t)),
\end{equation*}
where $K$ is a constant to be specified later. It is clear that $\widetilde{\theta}$ satisfies the following equation:
\begin{equation}\label{eq4.2}
  \partial_{t} \widetilde{\theta}+\mathbf{u}\cdot\nabla \widetilde{\theta}+\Lambda^{\alpha}
  \widetilde{\theta}=-K\mathcal{Y}'(t)\widetilde{\theta}.
\end{equation}
Thus we need to establish the following commutator estimate.

\begin{lemma}\label{le4.2}
Let $2\le p<\infty$. Then we have
\begin{equation}\label{eq4.3}
   \|[\mathbf{u}, \Delta_{j}]\cdot\nabla
   \widetilde{\theta}\|_{L^{p}}\lesssim 2^{-(1+\frac{2}{p}-\alpha)j}d_{j}\mathcal{Y}'(t)\|\widetilde{\theta}\|_{\dot{B}^{1+\frac{2}{p}-\alpha}_{p,1}}.
\end{equation}
\end{lemma}
\begin{proof}
We mention that throughout the paper,  the summation convention over repeated indices $i=1,2$ is used.
Thanks to Bony's paraproduct decomposition, by $\nabla\cdot\mathbf{u}=0$, we can decompose the commutator $[\mathbf{u}, \Delta_{j}]\cdot\nabla
   \widetilde{\theta}$ into the following terms:
\begin{equation}\label{eq4.4}
   [\mathbf{u}, \Delta_{j}]\cdot\nabla
   \widetilde{\theta}=[T_{u^{i}}, \Delta_{j}]\partial_{i}\widetilde{\theta}+T'_{\Delta_{j}\partial_{i}\theta}\widetilde{u}^{i}
   -\Delta_{j}T_{\partial_{i}\theta}\widetilde{u}^{i}-\Delta_{j}\partial_{i}R(\theta, \widetilde{u}^{i}),
\end{equation}
where $T'_{f}g:=T_{f}g+R(f,g)$. Applying the fact \eqref{eq2.2}, Lemmas \ref{le2.3} and
\ref{le2.4}, we infer from the first order Taylor's formula that
\begin{align*}
  \|[T_{u^{i}}, \Delta_{j}]\partial_{i}\widetilde{\theta}\|_{L^{p}}
  &\lesssim  \sum_{|j'-j|\le4}\|[S_{j'-1}u^{i}, \Delta_{j}]\Delta_{j'}\partial_{i}\widetilde{\theta}\|_{L^{p}}\nonumber\\
   &\lesssim \sum_{|j'-j|\le4}2^{-j'}\Big\|\int_{\mathbb{R}^{2}}\int_{0}^{1}h(y)\big(y\cdot\nabla S_{j'-1}u^{i}(x-2^{-j}\tau y)\big)
   \Delta_{j'}\partial_{i}\widetilde{\theta}(x-2^{-j}y)d\tau dy\Big\|_{L^{p}}\nonumber\\
   &\lesssim \sum_{|j'-j|\le4}2^{-j}\|\nabla S_{j'-1}u^{i}\|_{L^{\infty}}\|\Delta_{j'}\partial_{i}\widetilde{\theta}\|_{L^{p}}\nonumber\\
   &\lesssim \sum_{|j'-j|\le4}\sum_{k\le j'-2}2^{(1+\frac{2}{p})k}\|\Delta_{k}\theta\|_{L^{p}}\|\Delta_{j'}\widetilde{\theta}\|_{L^{p}}
  \nonumber\\
   &\lesssim 2^{-(1+\frac{2}{p}-\alpha)j}d_{j}\|\theta\|_{\dot{B}^{1+\frac{2}{p}}_{p,1}}
  \|\widetilde{\theta}\|_{\dot{B}^{1+\frac{2}{p}-\alpha}_{p,1}}
  \nonumber\\
  &\lesssim 2^{-(1+\frac{2}{p}-\alpha)j}d_{j}\mathcal{Y}'(t)\|\widetilde{\theta}\|_{\dot{B}^{1+\frac{2}{p}-\alpha}_{p,1}},
\end{align*}
where we have used the boundedness of Riesz operators in $L^{p}(\mathbb{R}^{2})$ with
$1<p<\infty$ to deduce that
$$
  \|\mathbf{u}\|_{L^{p}}\lesssim \|(-\mathcal{R}_{2}\theta,\mathcal{R}_{1}\theta)\|_{L^{p}}\lesssim \|\theta\|_{L^{p}}.
$$
Thanks to \eqref{eq2.2} again,  there exists a constant $N_{0}$ such that the second term in the right-hand side of \eqref{eq4.4} can be  rewritten as
$$
    T'_{\Delta_{j}\partial_{i}\theta}\widetilde{u}^{i}=\sum_{j'\ge j-N_{0}}S_{j'+2}\Delta_{j}\partial_{i}\theta\Delta_{j'}\widetilde{u}^{i},
$$
whence,
\begin{align*}
  \|T'_{\Delta_{j}\partial_{i}\theta}\widetilde{u}^{i}\|_{L^{p}}&\lesssim
  \sum_{j'\ge j-N_{0}}\|\Delta_{j'}\widetilde{u}^{i}\|_{L^{p}}\|\Delta_{j}\partial_{i}\theta\|_{L^{\infty}}\nonumber\\
  &\lesssim 2^{-(1+\frac{2}{p}-\alpha)j}\sum_{j'\ge j-N_{0}}2^{-(1+\frac{2}{p}-\alpha)(j'-j)}2^{(1+\frac{2}{p}-\alpha)j'}\|\Delta_{j'} \widetilde{\theta}\|_{L^{p}}
  2^{(1+\frac{2}{p})j}\|\Delta_{j}\theta\|_{L^{p}}
  \nonumber\\
   &\lesssim 2^{-(1+\frac{2}{p}-\alpha)j}d_{j}\|\theta\|_{\dot{B}^{1+\frac{2}{p}}_{p,1}}
  \|\widetilde{\theta}\|_{\dot{B}^{1+\frac{2}{p}-\alpha}_{p,1}}
  \nonumber\\
  &\lesssim 2^{-(1+\frac{2}{p}-\alpha)j}d_{j}\mathcal{Y}'(t)\|\widetilde{\theta}\|_{\dot{B}^{1+\frac{2}{p}-\alpha}_{p,1}}.
\end{align*}
On the other hand, the remaining two terms in the right-hand side of \eqref{eq4.4} can be estimated as
\begin{align*}
  \|\Delta_{j}T_{\partial_{i}\theta}\widetilde{u}^{i}\|_{L^{p}}&\lesssim
  \sum_{|j'-j|\le4}\|S_{j'-1}\partial_{i}\theta\|_{L^{\infty}}\|\Delta_{j'}\widetilde{u}^{i}\|_{L^{p}}\nonumber\\
  &\lesssim \sum_{|j'-j|\le4}\sum_{k\le j'-2}2^{(1+\frac{2}{p})k}\|\Delta_{k} \theta\|_{L^{p}}
 \|\Delta_{j'}\widetilde{\theta}\|_{L^{p}}\nonumber\\
   &\lesssim 2^{-(1+\frac{2}{p}-\alpha)j}d_{j}\|\theta\|_{\dot{B}^{1+\frac{2}{p}}_{p,1}}
  \|\widetilde{\theta}\|_{\dot{B}^{1+\frac{2}{p}-\alpha}_{p,1}}
  \nonumber\\
  &\lesssim 2^{-(1+\frac{2}{p}-\alpha)j}d_{j}\mathcal{Y}'(t)\|\widetilde{\theta}\|_{\dot{B}^{1+\frac{2}{p}-\alpha}_{p,1}};
\end{align*}
\begin{align*}
  \|\Delta_{j}\partial_{i}R(\theta, \widetilde{u}^{i})\|_{L^{p}}
  &\lesssim 2^{(1+\frac{2}{p})j}\sum_{j'\ge j-N_{0}}\|\Delta_{j'}\theta\|_{L^{p}}\|\widetilde{\Delta}_{j'}\widetilde{u}^{i}\|_{L^{p}}\nonumber\\
  &\lesssim 2^{(1+\frac{2}{p})j}\sum_{j'\ge j-N_{0}}2^{-(2+\frac{4}{p}-\alpha)j'}2^{(1+\frac{2}{p})j'}\|\Delta_{j'}\theta\|_{L^{p}}
  2^{(1+\frac{2}{p}-\alpha)j'}\|\widetilde{\Delta}_{j'}\widetilde{\theta}\|_{L^{p}}\nonumber\\
  &\lesssim 2^{-(1+\frac{2}{p}-\alpha)j}d_{j}\|\theta\|_{\dot{B}^{1+\frac{2}{p}}_{p,1}}
  \|\widetilde{\theta}\|_{\dot{B}^{1+\frac{2}{p}-\alpha}_{p,1}}
  \nonumber\\
  &\lesssim 2^{-(1+\frac{2}{p}-\alpha)j}d_{j}\mathcal{Y}'(t)\|\widetilde{\theta}\|_{\dot{B}^{1+\frac{2}{p}-\alpha}_{p,1}}.
\end{align*}
We finish the proof of Lemma
\ref{le4.2}.
\end{proof}

\medskip

\textbf{Proof  of Proposition \ref{pro4.1}}
 Applying the dyadic partition operator $\Delta_{j}$ to the equation \eqref{eq4.2}, then taking $L^{2}$ inner product with $|\Delta_{j}\widetilde{\theta}|^{p-2}\Delta_{j}\widetilde{\theta}$ yields that
\begin{align}\label{eq4.5}
  \frac{1}{p}\frac{d}{dt}\|\Delta_{j}\widetilde{\theta}\|_{L^{p}}^{p}&
   +\big(\Lambda^{\alpha}\Delta_{j}\widetilde{\theta}\big||\Delta_{j}\widetilde{\theta}|^{p-2}\Delta_{j}\widetilde{\theta}\big)=-\big([\mathbf{u}, \Delta_{j}]\Delta_{j}\widetilde{\theta}\big{|}|\Delta_{j}\widetilde{\theta}|^{p-2}\Delta_{j}\widetilde{\theta}\big)
  -K\mathcal{Y}'(t)\|\Delta_{j}\widetilde{\theta}\|_{L^{p}}^{p}\nonumber\\
  &\leq \|[\mathbf{u}, \Delta_{j}]\Delta_{j}\widetilde{\theta}\|_{L^{p}}\|\Delta_{j}\widetilde{\theta}\|_{L^{p}}^{p-1}
  -K\mathcal{Y}'(t)\|\Delta_{j}\widetilde{\theta}\|_{L^{p}}^{p},
\end{align}
where we have used the fact
\begin{equation*}
  \int_{\mathbb{R}^{2}}\mathbf{u}\cdot\Delta_{j}\widetilde{\theta}|\Delta_{j}\widetilde{\theta}|^{p-2}\Delta_{j}\widetilde{\theta}dx=0
\end{equation*}
due to the fact $\nabla\cdot \mathbf{u}=0$.
Thanks to \cite{ CMZ07, W05}, there exists a positive constant $\kappa$
so that
\begin{equation*}
  \int_{\mathbb{R}^{2}}\Lambda^{\alpha}\Delta_{j}\widetilde{\theta}|\Delta_{j}\widetilde{\theta}|^{p-2}\Delta_{j}\widetilde{\theta}dx\geq
  \kappa2^{\alpha j}\|\Delta_{j}\widetilde{\theta}\|_{L^{p}}^{p}.
\end{equation*}
Therefore, it follows from \eqref{eq4.5} that
\begin{align*}
  \frac{d}{dt}\|\Delta_{j}\widetilde{\theta}\|_{L^{p}}+\kappa2^{\alpha j}\|\Delta_{j}\widetilde{\theta}\|_{L^{p}}
   &\lesssim \|[\mathbf{u}, \Delta_{j}]\Delta_{j}\widetilde{\theta}\|_{L^{p}}
   -K\mathcal{Y}'(t)\|\Delta_{j}\widetilde{\theta}\|_{L^{p}}.
\end{align*}
Applying Lemma \ref{le4.2} leads directly to
\begin{align*}
  \frac{d}{dt}\|\Delta_{j}\widetilde{\theta}\|_{L^{p}}+\kappa2^{\alpha j}\|\Delta_{j}\widetilde{\theta}\|_{L^{p}}
   \leq C2^{-(1+\frac{2}{p}-\alpha)j}d_{j}\mathcal{Y}'(t)\|\widetilde{\theta}\|_{\dot{B}^{1+\frac{2}{p}-\alpha}_{p,1}}
   -K\mathcal{Y}'(t)\|\Delta_{j}\widetilde{\theta}\|_{L^{p}},
\end{align*}
which gives us to
\begin{align}\label{eq4.6}
  \frac{d}{dt}\|\widetilde{\theta}\|_{\dot{B}^{1+\frac{2}{p}-\alpha}_{p,1}}+\kappa\|\widetilde{\theta}\|_{\dot{B}^{1+\frac{2}{p}}_{p,1}}
   \leq C\mathcal{Y}'(t)\|\widetilde{\theta}\|_{\dot{B}^{1+\frac{2}{p}-\alpha}_{p,1}}
   -K\mathcal{Y}'(t)\|\widetilde{\theta}\|_{\dot{B}^{1+\frac{2}{p}-\alpha}_{p,1}}.
\end{align}
Hence, by choosing $K$ sufficiently large such that $K>C$, we see that
\begin{align*}
  \frac{d}{dt}\|\widetilde{\theta}\|_{\dot{B}^{1+\frac{2}{p}-\alpha}_{p,1}}+
  \kappa\|\widetilde{\theta}(t)\|_{\dot{B}^{1+\frac{2}{p}}_{p,1}}
   &\leq 0.
\end{align*}
We complete the proof of Proposition \ref{pro4.1}.

\medskip

Let $\ell$ be a real number and $2\le r\le p$,  set
$$
\mathcal{F}(t):= \|\theta(t)\|_{\dot{B}^{\ell}_{r,\infty}}.
$$
 We can further obtain the following result.

\begin{proposition}\label{pro4.3}
Under the assumptions of Proposition \ref{pro4.1}, if we further assume that
$\theta_{0}\in \dot{B}^{\ell}_{r,\infty}(\mathbb{R}^{2})$  with $2\le r\le p$, and
$$
  -1-\frac{2}{p}<\ell<\frac{2}{p},
$$
then there exist two positive constants $\kappa$ and $K$ such that for all $t\geq0$,
\begin{align}\label{eq4.7}
  \frac{d}{dt}(e^{-K\mathcal{Y}(t)}\mathcal{F}(t))+\kappa e^{-K\mathcal{Y}(t)}\|\theta(t)\|_{\dot{B}^{\ell+\alpha}_{r,\infty}}\leq0.
\end{align}
\end{proposition}
\begin{proof}
Applying the operator $\Delta_{j}\Lambda^{\ell}$ to the equation
\eqref{eq3.1},
then  taking $L^{2}$ inner product with $|\Delta_{j}\Lambda^{\ell}\theta|^{r-2}\Delta_{j}\Lambda^{\ell}\theta$ to the resultant,  we obtain  that
\begin{align*}
  \frac{1}{r}\frac{d}{dt}\|\Delta_{j}\Lambda^{\ell}\theta\|_{L^{r}}^{r}
  &+\big(\Lambda^{\alpha}\Delta_{j}\Lambda^{\ell}\theta\big{|}|\Delta_{j}\Lambda^{\ell}\theta|^{r-2}\Delta_{j}\Lambda^{\ell}\theta\big)=
  -\big(\Delta_{j}\Lambda^{\ell}(\mathbf{u}\cdot\nabla \theta)\big{|}|\Delta_{j}\Lambda^{\ell}\theta|^{r-2}\Delta_{j}\Lambda^{\ell}\theta\big)\nonumber\\
  &\leq \|\Delta_{j}\Lambda^{\ell}(\mathbf{u}\cdot\nabla \theta)\|_{L^{r}}
  \|\Delta_{j}\Lambda^{\ell}\theta\|_{L^{r}}^{r-1}.
\end{align*}
Thanks again to \cite{ CMZ07, W05}, there exists a positive constant $\kappa$
so that
\begin{equation*}
  \int_{\mathbb{R}^{2}}\Lambda^{\alpha}\Delta_{j}\Lambda^{\ell}\theta|\Delta_{j}\Lambda^{\ell}\theta|^{r-2}\Delta_{j}\Lambda^{\ell}\theta dx\geq
  \kappa2^{\alpha j}\|\Delta_{j}\Lambda^{\ell}\theta\|_{L^{r}}^{r}.
\end{equation*}
It follows that
\begin{align}\label{eq4.8}
  \frac{d}{dt}\|\Delta_{j}\Lambda^{\ell}\theta\|_{L^{r}}+\kappa2^{\alpha j}\|\Delta_{j}\Lambda^{\ell}\theta\|_{L^{r}}
   \lesssim \|\Delta_{j}\Lambda^{\ell}(\mathbf{u}\cdot\nabla \theta)\|_{L^{r}}.
\end{align}
Taking $l^{\infty}$ norm to \eqref{eq4.8}  and using Lemma \ref{le2.4},  we see that
\begin{align}\label{eq4.9}
  \frac{d}{dt}\|\theta\|_{\dot{B}^{\ell}_{r,\infty}}+\kappa\|\theta\|_{\dot{B}^{\ell+\alpha}_{r,\infty}}
   &\lesssim \|\mathbf{u}\cdot\nabla \theta\|_{\dot{B}^{\ell}_{r,\infty}}.
\end{align}
Thanks to  the Bony's paradifferential calculus and $\nabla\cdot\mathbf{u}=0$, we decompose
\begin{equation*}
   \mathbf{u}\cdot\nabla \theta=T_{u^{i}}\partial_{i}\theta+T_{\partial_{i}\theta}u^{i}+\partial_{i}R(u^{i}, \theta).
\end{equation*}
Applying  Lemmas \ref{le2.3} and
\ref{le2.4},  the above three terms can be estimated as follows:
\begin{align*}
  \|\Delta_{j}T_{u^{i}}\partial_{i}\theta\|_{L^{r}}
  &\lesssim  \sum_{|j'-j|\le4}\|S_{j'-1}u^{i}\|_{L^{\frac{pr}{p-r}}}\|\Delta_{j'}\partial_{i}\theta\|_{L^{p}}\nonumber\\
  &\lesssim \sum_{|j'-j|\le4}\sum_{k\le
   j'-2}2^{2(\frac{1}{r}-\frac{p-r}{pr})k}\|\Delta_{k}\theta\|_{L^{r}}2^{j'}\|\Delta_{j'}\theta\|_{L^{p}}\nonumber\\
   &\lesssim \sum_{|j'-j|\le4}\sum_{k\le j'-2}2^{(-\ell+\frac{2}{p}) k}2^{\ell k}\|\Delta_{k}\theta\|_{L^{r}}2^{j'}\|\Delta_{j'}\theta\|_{L^{p}}
  \nonumber\\
   &\lesssim 2^{-\ell j}\|\theta\|_{\dot{B}^{1+\frac{2}{p}}_{p,1}}
  \|\theta\|_{\dot{B}^{\ell}_{r,\infty}};
\end{align*}
\begin{align*}
  \|\Delta_{j}T_{\partial_{i}\theta}u^{i}\|_{L^{r}}&\lesssim  \sum_{|j'-j|\le4}\|S_{j'-1}\partial_{i}\theta\|_{L^{\infty}}\|\Delta_{j'}u^{i}\|_{L^{r}}\nonumber\\
   &\lesssim \sum_{|j'-j|\le4}\sum_{k\le j'-2}2^{(1+\frac{2}{p})k}\|\Delta_{k}\theta\|_{L^{p}}\|\Delta_{j'}\theta\|_{L^{r}}
  \nonumber\\
   &\lesssim 2^{-\ell j}\|\theta\|_{\dot{B}^{1+\frac{2}{p}}_{p,1}}
  \|\theta\|_{\dot{B}^{\ell}_{r,\infty}};
\end{align*}
\begin{align*}
  \|\Delta_{j}\partial_{i}R(u^{i},\theta)\|_{L^{r}}
  &\lesssim 2^{(1+\frac{2}{p})j}\sum_{j'\ge j-N_{0}}\|\Delta_{j'}u^{i}\|_{L^{p}}\|\widetilde{\Delta}_{j'}\theta\|_{L^{r}}\nonumber\\
  &\lesssim 2^{(1+\frac{2}{p})j}\sum_{j'\ge j-N_{0}}2^{-(1+\frac{2}{p}+\ell)j'}2^{(1+\frac{2}{p})j'}\|\Delta_{j'}\theta\|_{L^{p}}
  2^{\ell j'}\|\widetilde{\Delta}_{j'}\theta\|_{L^{r}}\nonumber\\
  &\lesssim 2^{-\ell j}\|\theta\|_{\dot{B}^{1+\frac{2}{p}}_{p,1}}
  \|\theta\|_{\dot{B}^{\ell}_{r,\infty}}.
\end{align*}
It is clear that the above estimates are verified due to $-1-\frac{2}{p}<\ell<\frac{2}{p}$, thus we obtain
\begin{align*}
   \|\mathbf{u}\cdot\nabla \theta\|_{\dot{B}^{\ell}_{r,\infty}}\lesssim \|\theta\|_{\dot{B}^{1+\frac{2}{p}}_{p,1}}
  \|\theta\|_{\dot{B}^{\ell}_{r,\infty}}.
\end{align*}
Therefore, we conclude by \eqref{eq4.9} that
\begin{align}\label{eq4.10}
  \frac{d}{dt}\|\theta\|_{\dot{B}^{\ell}_{r,\infty}}+\kappa\|\theta\|_{\dot{B}^{\ell+\alpha}_{r,\infty}}
   \leq C\|\theta\|_{\dot{B}^{1+\frac{2}{p}}_{p,1}}\|\theta\|_{\dot{B}^{\ell}_{r,\infty}}.
\end{align}
This yields \eqref{eq4.7} readily.  We complete the proof of Proposition \ref{pro4.3}.
\end{proof}

\subsection{Proof of Theorem \ref{th3.1}}

Now we present the proof of Theorem \ref{th3.1}.
For any real number $s$ such that $-\frac{2}{p}<s<1+\frac{2}{p}$,
we infer from Proposition \ref{pro4.3} that for all $t\geq0$,
\begin{align}\label{eq4.11}
  \|\theta(t)\|_{\dot{B}^{-s}_{r,\infty}}
  \leq C\|\theta_{0}\|_{\dot{B}^{-s}_{r,\infty}}\leq C_{0}.
\end{align}
Since $2\le r\leq p$,  the imbedding result in Lemma \ref{le2.4} yields that
$$
     \dot{B}^{-s}_{r,\infty}(\mathbb{R}^{2})\hookrightarrow\dot{B}^{-s-2(\frac{1}{r}-\frac{1}{p})}_{p,\infty}(\mathbb{R}^{2}),
$$
which together with \eqref{eq4.11} leads to for all $t\ge0$,
\begin{align}\label{eq4.12}
  \|\theta(t)\|_{\dot{B}^{-s-2(\frac{1}{r}-\frac{1}{p})}_{p,\infty}}
  \leq C_{0}.
\end{align}
 On the other hand, since $s>-\frac{2}{p}\ge\alpha-1-\frac{2}{r}$, the interpolation inequalities in Lemma \ref{le2.4} tell us  that
\begin{align*}
    \|\theta(t)\|_{\dot{B}^{1+\frac{2}{p}-\alpha}_{p,1}}&\leq C \|\theta(t)\|_{\dot{B}^{-s-2(\frac{1}{r}-\frac{1}{p})}_{p,\infty}}^{\frac{\alpha}{s+\frac{2}{r}+1}}
   \|\theta(t)\|_{\dot{B}^{1+\frac{2}{p}}_{p,1}}^{1-\frac{\alpha}{s+\frac{2}{r}+1}}.
\end{align*}
This together with \eqref{eq4.12}  implies that
\begin{align}\label{eq4.13}
   \|\theta(t)\|_{\dot{B}^{1+\frac{2}{p}}_{p,1}}&\geq C\|\theta(t)\|_{\dot{B}^{-s-2(\frac{1}{r}-\frac{1}{p})}_{p,\infty}}^{-\frac{\alpha}{s+\frac{2}{r}+1-\alpha}}
    \|\theta(t)\|_{\dot{B}^{1+\frac{2}{p}-\alpha}_{p,1}}^{1+\frac{\alpha}{s+\frac{2}{r}+1-\alpha}}
   \geq C\|\theta(t)\|_{\dot{B}^{1+\frac{2}{p}-\alpha}_{p,1}}^{1+\frac{\alpha}{s+\frac{2}{r}+1-\alpha}}.
\end{align}
Plugging \eqref{eq4.13}  into  \eqref{eq4.1}, by using the function $\mathcal{Y}(t)$ is positive along time evolution, we obtain
\begin{equation}\label{eq4.14}
   \frac{d}{dt}(e^{-K\mathcal{Y}(t)}\mathcal{E}(t))+C (e^{-K\mathcal{Y}(t)}\mathcal{E}(t))^{1+\frac{\alpha}{s+\frac{2}{r}+1-\alpha}}\leq 0.
\end{equation}
Solving this differential inequality directly, we have
\begin{equation*}
   \mathcal{E}(t)\leq e^{K\mathcal{Y}(t)}\left(\mathcal{E}(0)^{-\frac{\alpha}{s+\frac{2}{r}+1-\alpha}}+\frac{\alpha C t}{s+\frac{2}{r}+1-\alpha}\right)^{-\frac{s+\frac{2}{r}+1-\alpha}{\alpha}}.
\end{equation*}
 Since $\mathcal{Y}(t)$ is bounded by the initial data in Proposition \ref{pro4.1},  there exists a constant $C_{0}$ such that for all $t\geq0$,
\begin{equation}\label{eq4.15}
   \|\theta(t)\|_{\dot{B}^{1+\frac{2}{p}-\alpha}_{p,1}}\leq C_{0}\left(1+t\right)^{-\frac{s+\frac{2}{r}+1-\alpha}{\alpha}}.
\end{equation}
Notice that \eqref{eq4.15} gives in particular \eqref{eq3.3}  with $\ell=1+\frac{2}{p}-\alpha$.
Finally,  for any  $\ell\in[-s-2(\frac{1}{r}-\frac{1}{p}), 1+\frac{2}{p}-\alpha)$, by using interpolation inequalities in Lemma \ref{le2.4}, we see that
\begin{equation*}
     \|\theta(t)\|_{\dot{B}^{\ell}_{r,1}}\leq C\|\theta(t)\|_{\dot{B}^{-s-2(\frac{1}{r}-\frac{1}{p})}_{p,\infty}}^{\frac{1+\frac{2}{p}-\alpha-\ell}{s+\frac{2}{r}+1-\alpha}}
     \|\theta(t)\|_{\dot{B}^{1+\frac{2}{p}-\alpha}_{p,1}}^{\frac{\ell+s+2(\frac{1}{r}-\frac{1}{p})}{s+\frac{2}{r}+1-\alpha}},
\end{equation*}
which combining \eqref{eq4.12} and \eqref{eq4.15} implies that
\begin{equation*}
     \|\theta(t)\|_{\dot{B}^{\ell}_{r,1}}\leq C_{0}(1+t)^{-(\frac{\ell+s}{\alpha})-\frac{2}{\alpha}(\frac{1}{r}-\frac{1}{p})}.
\end{equation*}
We complete the proof of Theorem \ref{th3.1}, as desired.

\section{Proof of Theorem \ref{th3.2}}
Since the global well-posedness of the system \eqref{eq3.4} with small initial data $u_{0}\in\dot{B}^{-1+\frac{2}{p}}_{p,1}(\mathbb{R}^{2})$ has been proved by \cite{Z15}, it suffices to prove the temporal decay estimates \eqref{eq3.5} and \eqref{eq3.6}  in Theorem \ref{th3.2}.  Similarly, we first intend to establish the following two basic energy inequalities to the system  \eqref{eq3.4} in the framework of homogeneous Besov spaces.

\subsection{Basic energy inequalities}
Set
$$
  \mathcal{U}(t):= \|u(t)\|_{\dot{B}^{-1+\frac{2}{p}}_{p,1}}\ \ \ \text{and}\ \ \
   \mathcal{Z}(t):=\int_{0}^{t}\|u(\tau)\|_{\dot{B}^{\frac{2}{p}}_{p,1}}d\tau.
$$

\begin{proposition}\label{pro5.1}
Under the assumptions of Theorem \ref{th3.2},  let $u$ be the  unique global solution of  the system \eqref{eq3.4} corresponding to the initial data $u_{0}\in \dot{B}^{-1+\frac{2}{p}}_{p,1}(\mathbb{R}^{2})$.  Then there exist two constants $\kappa$ and $K$ such that
\begin{equation}\label{eq5.1}
   \frac{d}{dt}(e^{-K\mathcal{Z}(t)}\mathcal{U}(t))+\kappa e^{-K\mathcal{Z}(t)}\mathcal{Z}'(t)\leq 0.
\end{equation}
\end{proposition}

We first prove the following lemma. In order to do so, set
\begin{equation*}
  \widetilde{u}(t):=e^{-K\mathcal{Z}(t)}u(t),
\end{equation*}
where $K$ is a constant to be determined later.
\begin{lemma}\label{le5.2}
Let $2\leq p<\infty$.  Then
\begin{align}\label{eq5.2}
   \|\Delta_{j}(u\nabla(-\Delta)^{-1}\widetilde{u})\|_{L^{p}}\lesssim  2^{-\frac{2j}{p}}d_{j}\mathcal{Z}'(t)\|\widetilde{u}\|_{\dot{B}^{-1+\frac{2}{p}}_{p,1}}.
\end{align}
\end{lemma}
\begin{proof}
Thanks to the Bony's paraproduct decomposition, we have
\begin{equation*}
    u\nabla(-\Delta)^{-1}\widetilde{u}:=T_{u}\partial_{i}(-\Delta)^{-1}\widetilde{u}+T_{\partial_{i}(-\Delta)^{-1}\widetilde{u}}u+R(u, \partial_{i}(-\Delta)^{-1}\widetilde{u}).
\end{equation*}
Thus, by using  Lemmas \ref{le2.3} and \ref{le2.4}, we obtain that
\begin{align*}
   \|\Delta_{j}T_{u}\partial_{i}(-\Delta)^{-1}\widetilde{u}\|_{L^{p}} &
   \lesssim \sum_{|j-j'|\leq 4}\|S_{j'-1}u\|_{L^{\infty}}\|\Delta_{j'}\partial_{i}(-\Delta)^{-1}\widetilde{u}\|_{L^{p}}\nonumber\\
    &\lesssim \sum_{|j-j'|\leq 4}\sum_{k\leq j'-2}2^{\frac{2k}{p}}\|\Delta_{k}u\|_{L^{p}}2^{-j'}\|\Delta_{j'}\widetilde{u}\|_{L^{p}}\nonumber\\
  &\lesssim \sum_{|j-j'|\leq 4}2^{-j'}\|\Delta_{j'}\widetilde{u}\|_{L^{p}}\|u\|_{\dot{B}^{\frac{2}{p}}_{p,1}}\nonumber\\
  &\lesssim 2^{-\frac{2j}{p}}d_{j}\|u\|_{\dot{B}^{\frac{2}{p}}_{p,1}}\|\widetilde{u}\|_{\dot{B}^{-1+\frac{2}{p}}_{p,1}}\nonumber\\
  &\lesssim  2^{-\frac{2j}{p}}d_{j}\mathcal{Z}'(t)\|\widetilde{u}\|_{\dot{B}^{-1+\frac{2}{p}}_{p,1}};
\end{align*}
\begin{align*}
    \|\Delta_{j}T_{\partial_{i}(-\Delta)^{-1}\widetilde{u}}u\|_{L^{p}} &
    \lesssim \sum_{|j-j'|\leq 4}\|S_{j'-1}\partial_{i}(-\Delta)^{-1}\widetilde{u}\|_{L^{\infty}}\|\Delta_{j'}u\|_{L^{p}}\nonumber\\
     & \lesssim \sum_{|j-j'|\leq 4}\sum_{k\leq j'-2}2^{(-1+\frac{2}{p})k}\|\Delta_{k}\widetilde{u}\|_{L^{p}}\|\Delta_{j'}u\|_{L^{p}}\nonumber\\
  &\lesssim 2^{-\frac{2j}{p}}\sum_{|j-j'|\leq 4}2^{\frac{2j'}{p}}\|\Delta_{j'}u\|_{L^{p}}\|\widetilde{u}\|_{\dot{B}^{-1+\frac{2}{p}}_{p,1}}\nonumber\\
  &\lesssim  2^{-\frac{2j}{p}}d_{j}\|u\|_{\dot{B}^{\frac{2}{p}}_{p,1}}\|\widetilde{u}\|_{\dot{B}^{-1+\frac{2}{p}}_{p,1}}\nonumber\\
  &\lesssim 2^{-\frac{2j}{p}}d_{j}\mathcal{Z}'(t)\|\widetilde{u}\|_{\dot{B}^{-1+\frac{2}{p}}_{p,1}};
\end{align*}
\begin{align*}
    \|\Delta_{j}R(u, \partial_{i}(-\Delta)^{-1}\widetilde{u})\|_{L^{p}} &
    \lesssim 2^{\frac{2j}{p}} \sum_{j'\ge j-N_{0}}\|\Delta_{j'}u\|_{L^{p}}\|\widetilde{\Delta}_{j'}\partial_{i}(-\Delta)^{-1}\widetilde{u}\|_{L^{p}}\nonumber\\
     & \lesssim  2^{\frac{2j}{p}} \sum_{j'\ge j-N_{0}}2^{-\frac{4j'}{p}}2^{\frac{2j'}{p}}\|\Delta_{j'}u\|_{L^{p}}2^{(-1+\frac{2}{p})j'}\|\widetilde{\Delta}_{j'}\widetilde{u}\|_{L^{p}}\nonumber\\
  &\lesssim  2^{-\frac{2j}{p}}d_{j}\|u\|_{\dot{B}^{\frac{2}{p}}_{p,1}}\|\widetilde{u}\|_{\dot{B}^{-1+\frac{2}{p}}_{p,1}}\nonumber\\
  &\lesssim 2^{-\frac{2j}{p}}d_{j}\mathcal{Z}'(t)\|\widetilde{u}\|_{\dot{B}^{-1+\frac{2}{p}}_{p,1}}.
\end{align*}
We finish the proof of Lemma \ref{le5.2}.
\end{proof}

\medskip

\textbf{Proof  of Proposition \ref{pro5.1}}
We first observe that $\widetilde{u}$ satisfies the following equation:
\begin{equation}\label{eq5.3}
  \partial_{t} \widetilde{u}+\Lambda\widetilde{u}+\nabla\cdot(u\nabla\widetilde{\psi})=-K\mathcal{Z}'(t)\widetilde{u},
\end{equation}
where $\widetilde{\psi}(t):= e^{-K\mathcal{Z}(t)}\psi(t)=(-\Delta)^{-1}\widetilde{u}(t)$.
Then applying the homogeneous localization operator $\Delta_{j}$ to the equation \eqref{eq5.3} and taking $L^{2}$ inner product with $|\Delta_{j}\widetilde{u}|^{p-2}\Delta_{j}\widetilde{u}$, we see that
\begin{align}\label{eq5.4}
  \frac{1}{p}\frac{d}{dt}\|\Delta_{j}\widetilde{u}\|_{L^{p}}^{p}
  &+\big(\Lambda\Delta_{j}\widetilde{u}\big||\Delta_{j}\widetilde{u}|^{p-2}\Delta_{j}\widetilde{u}\big)=
  -\big(\Delta_{j}\nabla\cdot(u\nabla\widetilde{\psi})\big{|}|\Delta_{j}\widetilde{u}|^{p-2}\Delta_{j}\widetilde{u}\big)
  -K\mathcal{Z}'(t)\|\Delta_{j}\widetilde{u}\|_{L^{p}}^{p}\nonumber\\
  &\leq \|\Delta_{j}\nabla\cdot(u\nabla\widetilde{\psi})\|_{L^{p}}\|\Delta_{j}\widetilde{u}\|_{L^{p}}^{p-1}-K\mathcal{Z}'(t)\|\Delta_{j}\widetilde{u}\|_{L^{p}}^{p}.
\end{align}
Thanks to \cite{ CMZ07, W05}, there exists a positive constant $\kappa$
so that
\begin{equation*}
  \int_{\mathbb{R}^{2}}\Lambda\Delta_{j}\widetilde{u}|\Delta_{j}\widetilde{u}|^{p-2}\Delta_{j}\widetilde{u}dx\geq
  \kappa2^{j}\|\Delta_{j}\widetilde{u}\|_{L^{p}}^{p}.
\end{equation*}
It follows that
\begin{align*}
  \frac{d}{dt}\|\Delta_{j}\widetilde{u}\|_{L^{p}}+\kappa2^{j}\|\Delta_{j}\widetilde{u}\|_{L^{p}}
   &\lesssim
 \|\Delta_{j}\nabla\cdot(u\nabla\widetilde{\psi})\|_{L^{p}}-K\mathcal{Z}'(t)\|\Delta_{j}\widetilde{u}\|_{L^{p}}\\
  &\lesssim
  2^{j}\|\Delta_{j}(u\nabla(-\Delta)^{-1}\widetilde{u})\|_{L^{p}}-K\mathcal{Z}'(t)\|\Delta_{j}\widetilde{u}\|_{L^{p}}.
\end{align*}
Lemma \ref{le5.2}  gives us to
\begin{align*}
  \frac{d}{dt}\|\Delta_{j}\widetilde{u}\|_{L^{p}}+\kappa2^{j}\|\Delta_{j}\widetilde{u}\|_{L^{p}}
   &\lesssim 2^{(1-\frac{2}{p})j}d_{j}\mathcal{Z}'(t)\|\widetilde{u}\|_{\dot{B}^{-1+\frac{2}{p}}_{p,1}}-K\mathcal{Z}'(t)\|\Delta_{j}\widetilde{u}\|_{L^{p}},
\end{align*}
which implies directly that
\begin{align}\label{eq5.5}
  \frac{d}{dt}\|\widetilde{u}(t)\|_{\dot{B}^{-1+\frac{2}{p}}_{p,1}}+\kappa\|\widetilde{u}(t)\|_{\dot{B}^{\frac{2}{p}}_{p,1}}
   &\leq C\mathcal{Z}'(t)\|\widetilde{u}(t)\|_{\dot{B}^{-1+\frac{2}{p}}_{p,1}}-K\mathcal{Z}'(t)\|\widetilde{u}(t)\|_{\dot{B}^{-1+\frac{2}{p}}_{p,1}}.
\end{align}
By choosing $K$ sufficiently large such that $K>C$, we see that
\begin{align*}
   \frac{d}{dt}\|\widetilde{u}(t)\|_{\dot{B}^{-1+\frac{2}{p}}_{p,1}}+
  \kappa\|\widetilde{u}(t)\|_{\dot{B}^{\frac{2}{p}}_{p,1}}
   &\leq 0.
\end{align*}
The proof of Proposition \ref{pro5.1} is complete.

\medskip

Let $\ell$ be a real number and $2\le r\le p$.
Define
$$
\mathcal{V}(t):= \|u(t)\|_{\dot{B}^{\ell}_{r,\infty}}.
$$
\begin{proposition}\label{pro5.3}
Under the assumptions of Proposition \ref{pro5.1}, if we further assume that
$u_{0}\in \dot{B}^{\ell}_{r,\infty}(\mathbb{R}^{2})$ with $2\le r\leq p$, and
$$
   -1-\frac{2}{p}<\ell<-1+\frac{2}{p},
$$
then there exist two positive constants $\kappa$ and $K$ such that for all $t\geq0$,  we have
\begin{align}\label{eq5.6}
  \frac{d}{dt}(e^{-K\mathcal{Z}(t)}\mathcal{V}(t))+ \kappa e^{-K\mathcal{Z}(t)}\|u(t)\|_{\dot{B}^{\ell+1}_{r,\infty}}\leq0.
\end{align}
\end{proposition}
\begin{proof}
Applying the homogeneous localization operator $\Delta_{j}\Lambda^{\ell}$ to the first equation of
\eqref{eq3.4},
then  taking $L^{2}$ inner product with $|\Delta_{j}\Lambda^{\ell}u|^{r-2}\Delta_{j}\Lambda^{\ell}u$ to the resultant,  we obtain  that
\begin{align*}
  \frac{1}{r}\frac{d}{dt}\|\Delta_{j}\Lambda^{\ell}u\|_{L^{r}}^{r}
  &+\big(\Lambda\Delta_{j}\Lambda^{\ell}u\big{|}|\Delta_{j}\Lambda^{\ell}u|^{r-2}\Delta_{j}\Lambda^{\ell}u\big)=
  -\big(\Delta_{j}\Lambda^{\ell}\nabla\cdot(u\nabla \psi)\big{|}|\Delta_{j}\Lambda^{\ell}u|^{r-2}\Delta_{j}\Lambda^{\ell}u\big)\nonumber\\
  &\leq \|\Delta_{j}\Lambda^{\ell}\nabla\cdot(u\nabla \psi)\|_{L^{r}}
  \|\Delta_{j}\Lambda^{\ell}u\|_{L^{r}}^{r-1},
\end{align*}
Thanks again to \cite{ CMZ07, W05}, there exists a positive constant $\kappa$
so that
\begin{equation*}
  \int_{\mathbb{R}^{2}}\Lambda\Delta_{j}\Lambda^{\ell}u|\Delta_{j}\Lambda^{\ell}u|^{r-2}\Delta_{j}\Lambda^{\ell}u dx\geq
  \kappa2^{ j}\|\Delta_{j}\Lambda^{\ell}u\|_{L^{r}}^{r}.
\end{equation*}
It follows that
\begin{align}\label{eq5.7}
  \frac{d}{dt}\|\Delta_{j}\Lambda^{\ell}u\|_{L^{r}}+\kappa2^{ j}\|\Delta_{j}\Lambda^{\ell}u\|_{L^{r}}
   \lesssim \|\Delta_{j}\Lambda^{\ell}\nabla\cdot(u\nabla \psi)\|_{L^{r}}\lesssim \|\Delta_{j}\Lambda^{\ell}\nabla\cdot(u\nabla (-\Delta)^{-1}u)\|_{L^{r}}.
\end{align}
Taking $l^{\infty}$ norm to \eqref{eq5.7}  and using Lemma \ref{le2.4},  we see that
\begin{align}\label{eq5.8}
  \frac{d}{dt}\|u\|_{\dot{B}^{\ell}_{r,\infty}}+\kappa\|u\|_{\dot{B}^{\ell+1}_{r,\infty}}
   \lesssim \|\nabla\cdot(u\nabla (-\Delta)^{-1}u)\|_{\dot{B}^{\ell}_{r,\infty}}\approx \|u\nabla (-\Delta)^{-1}u\|_{\dot{B}^{\ell+1}_{r,\infty}}.
\end{align}
In order to calculate the right-hand side term of \eqref{eq5.8}, we resort  the Bony's  paraproduct decomposition to deduce that
\begin{equation*}
    u\nabla(-\Delta)^{-1}u:=T_{u}\partial_{i}(-\Delta)^{-1}u+T_{\partial_{i}(-\Delta)^{-1}u}u+R(u, \partial_{i}(-\Delta)^{-1}u).
\end{equation*}
Applying  Lemmas \ref{le2.3} and
\ref{le2.4}, we can estimate the above three terms as follows:
\begin{align*}
  \|\Delta_{j}T_{u}\partial_{i}(-\Delta)^{-1}u\|_{L^{r}}
  &\lesssim  \sum_{|j'-j|\le4}\|S_{j'-1}u\|_{L^{\infty}}\|\Delta_{j'}\partial_{i}(-\Delta)^{-1}u\|_{L^{r}}\nonumber\\
   &\lesssim \sum_{|j'-j|\le4}\sum_{k\le j'-2}2^{\frac{2k}{p}}\|\Delta_{k}u\|_{L^{p}}2^{-j'}\|\Delta_{j'}u\|_{L^{r}}
  \nonumber\\
   &\lesssim 2^{-(\ell+1) j}\|u\|_{\dot{B}^{\frac{2}{p}}_{p,1}}
  \|u\|_{\dot{B}^{\ell}_{r,\infty}};
\end{align*}
\begin{align*}
  \|\Delta_{j}T_{\partial_{i}(-\Delta)^{-1}u}u\|_{L^{r}}&\lesssim
   \sum_{|j'-j|\le4}\|S_{j'-1}\partial_{i}(-\Delta)^{-1}u\|_{L^{\frac{pr}{p-r}}}\|\Delta_{j'}u\|_{L^{p}}\nonumber\\
  &\lesssim  \sum_{|j'-j|\le4}\sum_{k\le j'-2}2^{-k+2(\frac{1}{r}-\frac{p-r}{pr})k}\|\Delta_{k}u\|_{L^{r}}\|\Delta_{j'}u\|_{L^{p}}\nonumber\\
   &\lesssim \sum_{|j'-j|\le4}\sum_{k\le j'-2}2^{-(1+\ell-\frac{2}{p})k}2^{\ell k}\|\Delta_{k}u\|_{L^{r}}\|\Delta_{j'}u\|_{L^{p}}
  \nonumber\\
   &\lesssim 2^{-(1+\ell)j}\|u\|_{\dot{B}^{\frac{2}{p}}_{p,1}}
  \|u\|_{\dot{B}^{\ell}_{r,\infty}};
\end{align*}
\begin{align*}
  \|\Delta_{j}R(u,\partial_{i}(-\Delta)^{-1}u)\|_{L^{r}}
  &\lesssim 2^{\frac{2j}{p}}\sum_{j'\ge j-N_{0}}\|\Delta_{j'}u\|_{L^{p}}\|\widetilde{\Delta}_{j'}\partial_{i}(-\Delta)^{-1}u\|_{L^{r}}\nonumber\\
  &\lesssim 2^{\frac{2j}{p}}\sum_{j'\ge j-N_{0}}2^{-(1+\frac{2}{p}+\ell)j'}2^{\frac{2j'}{p}}\|\Delta_{j'}u\|_{L^{p}}
  2^{\ell j'}\|\widetilde{\Delta}_{j'}u\|_{L^{r}}\nonumber\\
  &\lesssim 2^{-(\ell+1) j}\|u\|_{\dot{B}^{\frac{2}{p}}_{p,1}}
  \|u\|_{\dot{B}^{\ell}_{r,\infty}}.
\end{align*}
The above estimates are verified due to $-1-\frac{2}{p}<\ell<-1+\frac{2}{p}$, which yields that
\begin{align*}
   \|u\nabla(-\Delta)^{-1}u\|_{\dot{B}^{\ell+1}_{r,\infty}}\lesssim \|u\|_{\dot{B}^{\frac{2}{p}}_{p,1}}
  \|u\|_{\dot{B}^{\ell}_{r,\infty}}.
\end{align*}
Therefore, returning back to \eqref{eq5.8}, we conclude that
\begin{align}\label{eq5.9}
  \frac{d}{dt}\|u\|_{\dot{B}^{\ell}_{r,\infty}}+\kappa\|u\|_{\dot{B}^{\ell+1}_{r,\infty}}
   \leq C\|u\|_{\dot{B}^{\frac{2}{p}}_{p,1}}\|u\|_{\dot{B}^{\ell}_{r,\infty}}.
\end{align}
 The proof of Proposition \ref{pro5.3} is accomplished.
\end{proof}

\subsection{Proof of Theorem \ref{th3.2}}

Now we present the proof of Theorem \ref{th3.2}.
By Proposition \ref{pro5.3}, for any $s>0$ such that $1-\frac{2}{p}<s<1+\frac{2}{p}$,
we see that for all $t\geq0$,
\begin{align}\label{eq5.10}
  \|u(t)\|_{\dot{B}^{-s}_{r,\infty}}
  \leq C\|u_{0}\|_{\dot{B}^{-s}_{r,\infty}}\leq C_{0}.
\end{align}
Since $2\le r\le  p$,  we infer from the imbedding in Lemma \ref{le2.4} that
$$
     \dot{B}^{-s}_{r,\infty}(\mathbb{R}^{2})\hookrightarrow\dot{B}^{-s-2(\frac{1}{r}-\frac{1}{p})}_{p,\infty}(\mathbb{R}^{2}),
$$
which together with \eqref{eq5.10} leads to for all $t\ge0$,
\begin{align}\label{eq5.11}
  \|u(t)\|_{\dot{B}^{-s-2(\frac{1}{r}-\frac{1}{p})}_{p,\infty}}
  \leq C_{0}.
\end{align}
 On the other hand, since $s>1-\frac{2}{p}\ge 1-\frac{2}{r}$,  by interpolation inequalities in Lemma \ref{le2.4}, we have
\begin{align*}
    \|u(t)\|_{\dot{B}^{-1+\frac{2}{p}}_{p,1}}&\leq C \|u(t)\|_{\dot{B}^{-s-2(\frac{1}{r}-\frac{1}{p})}_{p,\infty}}^{\frac{1}{s+\frac{2}{r}}}
   \|u(t)\|_{\dot{B}^{\frac{2}{p}}_{p,1}}^{1-\frac{1}{s+\frac{2}{r}}}.
\end{align*}
This together with \eqref{eq5.11}  implies that
\begin{align}\label{eq5.12}
   \|u(t)\|_{\dot{B}^{\frac{2}{p}}_{p,1}}&\geq C\|u(t)\|_{\dot{B}^{-s-2(\frac{1}{r}-\frac{1}{p})}_{p,\infty}}^{-\frac{1}{s+\frac{2}{r}-1}}
    \|u(t)\|_{\dot{B}^{-1+\frac{2}{p}}_{p,1}}^{1+\frac{1}{s+\frac{2}{r}-1}}
   \geq C\|u(t)\|_{\dot{B}^{-1+\frac{2}{p}}_{p,1}}^{1+\frac{1}{s+\frac{2}{r}-1}}=C\mathcal{U}(t)^{1+\frac{1}{s+\frac{2}{r}-1}}.
\end{align}
Plugging \eqref{eq5.12}  into  \eqref{eq5.1}, by using the function $\mathcal{Z}(t)$ is positive along time evolution, we obtain
\begin{equation}\label{eq5.13}
   \frac{d}{dt}(e^{-K\mathcal{Z}(t)}\mathcal{U}(t))+C (e^{-K\mathcal{Z}(t)}\mathcal{U}(t))^{1+\frac{1}{s+\frac{2}{r}-1}}\leq 0.
\end{equation}
Solving this differential inequality directly, we obtain
\begin{equation*}
   \mathcal{U}(t)\leq e^{K\mathcal{Z}(t)}\left(\mathcal{U}(0)^{-\frac{1}{s+\frac{2}{r}-1}}+\frac{C t}{s+\frac{2}{r}-1}\right)^{-(s+\frac{2}{r}-1)}.
\end{equation*}
 Since $\mathcal{Z}(t)$ is bounded by the initial data in Proposition \ref{pro5.1},  there exists a constant $C_{0}$ such that for all $t\geq0$,
\begin{equation}\label{eq5.14}
   \|u(t)\|_{\dot{B}^{-1+\frac{2}{p}}_{p,1}}\leq C_{0}\left(1+t\right)^{-(s+\frac{2}{r}-1)}.
\end{equation}
Notice that \eqref{eq5.14} gives in particular \eqref{eq3.6}  with $\ell=-1+\frac{2}{p}$.
Finally,  for any  $\ell\in[-s-2(\frac{1}{r}-\frac{1}{p}), -1+\frac{2}{p})$, by using interpolation inequalities in Lemma \ref{le2.4}, we see that
\begin{equation*}
     \|u(t)\|_{\dot{B}^{\ell}_{r,1}}\leq C\|u(t)\|_{\dot{B}^{-s-2(\frac{1}{r}-\frac{1}{p})}_{p,\infty}}^{\frac{\frac{2}{p}-\ell-1}{s+\frac{2}{r}-1}}
     \|u(t)\|_{\dot{B}^{-1+\frac{2}{p}}_{p,1}}^{\frac{\ell+s+2(\frac{1}{r}-\frac{1}{p})}{s+\frac{2}{r}-1}},
\end{equation*}
which combining \eqref{eq5.11} and \eqref{eq5.14} implies that
\begin{equation*}
     \|u(t)\|_{\dot{B}^{\ell}_{r,1}}\leq C_{0}(1+t)^{-(\ell+s)-2(\frac{1}{r}-\frac{1}{p})}.
\end{equation*}
We complete the proof of Theorem \ref{th3.2}, as desired.

\medskip
\medskip

\noindent\textbf{Acknowledgments.} The author is deeply grateful to Professor Qiao Liu for his warm communication and valuable suggestions. This paper is partially supported by the National Natural Science Foundation of China (11371294), the Fundamental Research Funds for the Central Universities (2014YB031) and the Fundamental Research Project of Natural Science in Shaanxi Province--Young Talent Project (2015JQ1004).

\end{document}